\documentclass[letterpaper, 10 pt, conference]{ieeeconf}  

\IEEEoverridecommandlockouts                              
\usepackage{graphicx, , amsmath}
\usepackage{amsfonts, amssymb}

\IEEEoverridecommandlockouts
\usepackage{mdwmath}
\usepackage{texdraw}%
\usepackage{psfrag}%
\usepackage{accents}%
\usepackage{color}%
\usepackage{cite}
\usepackage{graphics}
\usepackage[tight,footnotesize]{subfigure}

\newtheorem{theorem}{Theorem}
\newtheorem{lemma}{Lemma}

\newtheorem{remark}{Remark}

\title{\LARGE \bf
On Optimal Jamming Over an Additive Noise Channel
}

\author{  Emrah Akyol, Kenneth Rose, and  Tamer Ba\c{s}ar
\thanks{Emrah Akyol and Kenneth Rose are supported in part by the NSF under grants 
CCF-1016861 and CCF-1118075.}
\thanks{Emrah Akyol and Kenneth Rose are with Department of Electrical and Computer Engineering, University of California at Santa Barbara, CA 93106, USA
         {\tt\small \{eakyol, rose\}@ece.ucsb.edu}}%
\thanks{Tamer Ba\c{s}ar is with Department of Electrical and Computer Engineering, University of Illinois at Urbana-Champaign, 1308
        West Main Street, Urbana, IL 61801, USA
        {\tt\small basar1@illinois.edu}}%
}

\begin{document}

\maketitle
\thispagestyle{empty}
\pagestyle{empty}

\begin{abstract}
 This paper considers the problem of optimal zero-delay jamming over an additive noise channel. Early work had already solved this problem for a Gaussian source and channel. Building on a sequence of recent results on conditions for linearity of optimal estimation, and of optimal mappings in source-channel coding, we derive the saddle-point solution to the jamming problem for general sources and channels, without recourse to Gaussian assumptions. We show that linearity conditions play a pivotal role in jamming, in the sense that the optimal jamming strategy is to effectively force both transmitter and receiver to default to linear mappings, i.e., the jammer ensures, whenever possible, that the transmitter and receiver cannot benefit from non-linear strategies. This result is shown to subsume the known result for Gaussian source and channel. We analyze conditions and general settings where such unbeatable strategy can indeed be achieved by the jammer. Moreover, we provide the procedure to approximate optimal jamming in the remaining (source-channel) cases where the jammer cannot impose linearity on the transmitter and the receiver.
\end{abstract}

\section{Introduction}

The interaction between communication and control has been an important research area for decades. We consider the problem of optimal jamming, by a power constrained agent, over additive noise channel. This problem was already solved in \cite{basar1983gaussian,basar1985complete} for Gaussian source and channel. The saddle point solution to this zero-sum game, derived for Gaussian source-channel pair, involves randomized linear mapping for the transmitter and generating independent, Gaussian noise as the jammer output and a linear decoder. In this paper, by leveraging the recent results on conditions for linearity of optimal estimation and communication mappings, \cite{akyol2012conditions, mapping}, we extend the analysis to non-Gaussian sources and channels.  The contributions of this paper are: 
\begin{itemize}
\item We show that linearity is essential in the jamming problem. The jammer, whenever possible, forces the transmitter and the receiver to be linear. In the Gaussian source-channel setting, this effect corresponds to generating a Gaussian jamming noise, while in general the optimal jamming noise is not Gaussian. 
\item We derive the necessary and sufficient condition (called the ``matching condition") on the jamming noise density to ensure linearity of the optimal transmitter and the receiver.
\item Based on the matching condition, we derive asymptotic (in channel signal-to-noise ratio (CSNR)),  optimal jamming strategies. 
\item We present a numerical method to approximate the optimal jammer strategy, in cases where a matching jamming density does not exist and it cannot force the transmitter and receiver to be exactly linear. 
\end{itemize}

  \begin{figure}
\centering
\includegraphics[scale=0.4]{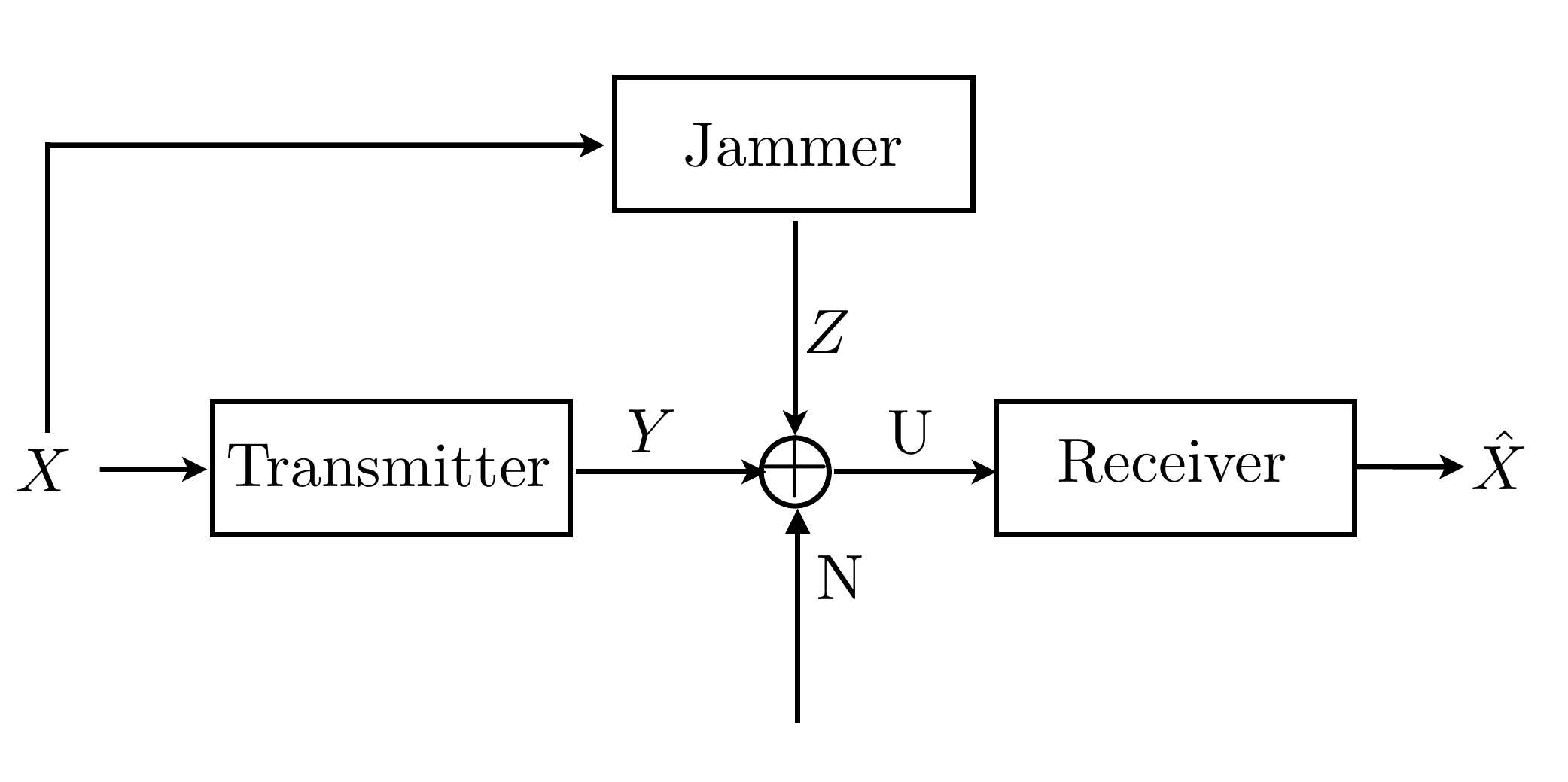}
\caption{The jamming problem.}
\label{jammingfig}
\end{figure}

This paper is organized as follows. In Section II, we present the problem definition and preliminaries. In Section III, we review the prior results related to jamming, estimation and communication problems. In Section IV, we derive the linearity result, which leads to our main result in Section V. In Section VI, we study the implications of the main result, and in Section VII, we present a procedure to approximate the optimal jamming density in the non-matching case. Finally, we discuss the future directions in Section VIII.

\section{Problem Definition}
Let $\mathbb R$ and $\mathbb R^+$ denote the respective sets of real numbers and positive real numbers. Let  $\mathbb E(\cdot)$,  $\mathbb P(\cdot)$ and $\ast$ denote the expectation, probability and convolution operators, respectively. Let $Bern(p)$ denote the Bernoulli random variable, taking values in $\{-1,1\}$ with probability $\{p,1-p\}$. The Gaussian density with mean $ \mu$ and variance  $\sigma^2$ is denoted as $\mathcal N( \mu,\sigma^2)$. Let $f^{'} (x)= \frac{d f(x)}{d x}$ denote the first order derivative of the function $f(\cdot)$.  Let $\delta(\cdot,\cdot)$ denote the Kronecker delta function. All logarithms in the paper are natural logarithms and may in general be complex, and the integrals are, in general, Lebesgue integrals.  Let us define the set $\cal S$ as the set of  Borel measurable $\mathbb R \rightarrow \mathbb R$ mappings. 

We consider the general communication system whose block diagram is shown in Figure 1. A scalar zero mean\footnote{The zero mean assumption is not necessary, but it considerably simpliÞes the notation. Therefore, it is made throughout the paper.}  source ${ X} \in\mathbb R$ is mapped into ${ Y}\in  \mathbb R$ by function ${ g_T (\cdot)}\in {{\cal {S}}}$ and transmitted over an additive noise channel. The adversary  receives the same signal $X$ and generates the jamming signal $Z$ through function ${g_A(\cdot)}\in {{\cal {S}}}$  which is added to the channel output, and aims to compromise the transmission of the source. The received signal $ {U}= Y+ Z+ N  $ is  mapped by the decoder to an estimate $ {\hat X}$ via function ${ h(\cdot)} \in {{\cal {S}}}$. The zero mean  noise $ N$ is assumed to be independent of the source $X$. The  source density is denoted $f_X(\cdot)$ and the noise density is $f_N(\cdot)$ with characteristic functions $F_X( \omega)$ and $F_N( \omega)$, respectively. We assume that the source and the noise variances are finite, i.e.,  $\sigma_X^2=\mathbb E\{X^2\} < \infty$ and  $\sigma_N^2=\mathbb E\{N^2\}<\infty$.

The overall cost, measured as the mean squared error (MSE) between  source $X$ and its estimate at the decoder $\hat X$, is a function of the transmitter, jammer and the receiver mappings: 
\begin{equation}
J(g_T(\cdot),g_A(\cdot),h(\cdot))=\ \mathbb E \{(X-{\hat X})^2\}.
\end{equation}

Transmitter $g_T(\cdot): \mathbb R\rightarrow\mathbb  R$ and receiver $h(\cdot): \mathbb R\rightarrow\mathbb  R$ seek to minimize this cost while the adversary (jammer) seeks to maximize it by appropriate  choice of $g_A(\cdot): \mathbb R\rightarrow \mathbb R$. Power constraints must be satisfied by the transmitter
\begin{equation}
\mathbb E\{g_T^2(X)\} \leq P_T,
\end{equation}
 and jammer
 \begin{equation}
\mathbb E\{g_A^2(X)\} \leq P_A.
\end{equation}
 
The conflict of interest underlying this problem implies that the optimal transmitter-receiver-adversarial policy is the saddle point solution ($g_T^*(\cdot),g_A^*(\cdot),h^*(\cdot)$) satisfying the set of inequalities
\begin{equation}
\label{saddle}
J(g_T^*,g_A^{},h^*) \leq J(g_T^{*},g_A^{*},h^*) \leq J(g_T^{},g_A^{*},h) .
\end{equation} 


\section{Prior Work}
The jamming problem, in the form defined here, was studied in \cite{basar1983gaussian,basar1985complete}, for Gaussian sources and channels. The problem of interest is intrinsically connected to the fundamental problems of estimation theory and the theory of zero-delay source-channel coding. In particular, conditions for linearity of optimal estimation\cite{akyol2012conditions}  and optimal mappings in communications \cite{mapping} are relevant to our problem here. We start with the estimation problem. 

\subsection{Estimation Problem}

\begin{figure}
\centering
\includegraphics[scale=0.5]{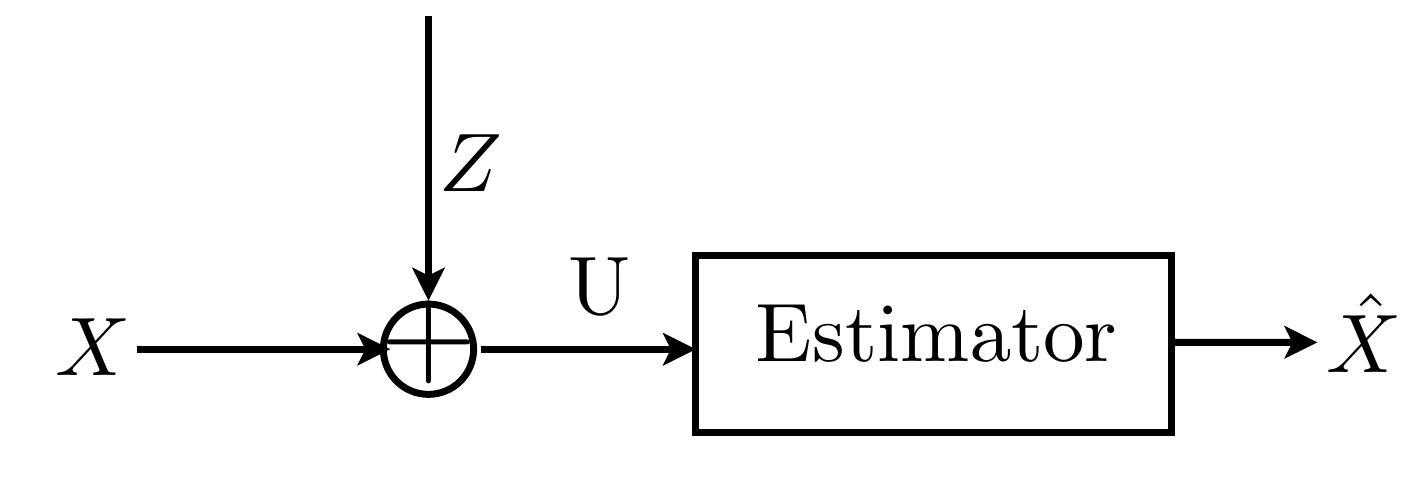}
\caption{The estimation problem.}
\label{est}
\end{figure}

Consider the setting in Figure 2. The estimator receives $U$, the noisy version of the source $X$ and generates the estimate $\hat X$ by the function $h: \mathbb R\rightarrow \mathbb R$ such that MSE, $\mathbb E\{(X-\hat X)^2\}$ is minimized. It is well known that, when a Gaussian source is contaminated with Gaussian noise, a linear estimator minimizes MSE. Recent work \cite{akyol2012conditions} analyzed, more generally, the conditions for linearity of optimal estimators. Given a noise (or source) distribution, and a specified signal to noise ratio (SNR),  conditions for existence and uniqueness of a source (or noise) distribution for which the optimal estimator is linear were derived. Also, the asymptotic linearity of the optimal estimators was shown for low SNR if the channel is Gaussian regardless of the source and, vice versa, for high SNR if the source is Gaussian regardless of the channel.

Here, we present the basic result pertaining to the jamming problem considered here. Specifically, we present the necessary and sufficient condition for source and channel distributions such that the linear estimator  $h(U)=\frac{\kappa}{\kappa+1} U$ is optimal where $\kappa=\frac{\sigma_X^2}{\sigma_Z^2}$ is the SNR.

\begin{theorem}[\cite{akyol2012conditions}]
Given SNR level $\kappa$, and noise $Z$ with  characteristic function $F_Z(\omega)$, there exists a source $X$ for which the optimal estimator is linear {\em if and only if} 
\begin{equation}F_X(\omega)= F_Z^{\kappa}(\omega). \end{equation}
\end{theorem}

Given  a valid characteristic function  $F_Z(\omega)$,  and for some $\kappa \in \mathbb R^+$, the function $F_Z^{\kappa}(\omega)$ may or may not be a valid characteristic function, which determines the existence of a matching source. For example, matching is guaranteed for integer $\kappa$ and it is also guaranteed for infinitely divisible $Z$. More comprehensive discussion of the conditions on $\kappa$ and  $F_Z(\omega)$ for $F_Z^{\kappa}(\omega)$ to be a valid characteristic function can be found in \cite{akyol2012conditions}.
\begin{figure}
\centering
\includegraphics[scale=0.4]{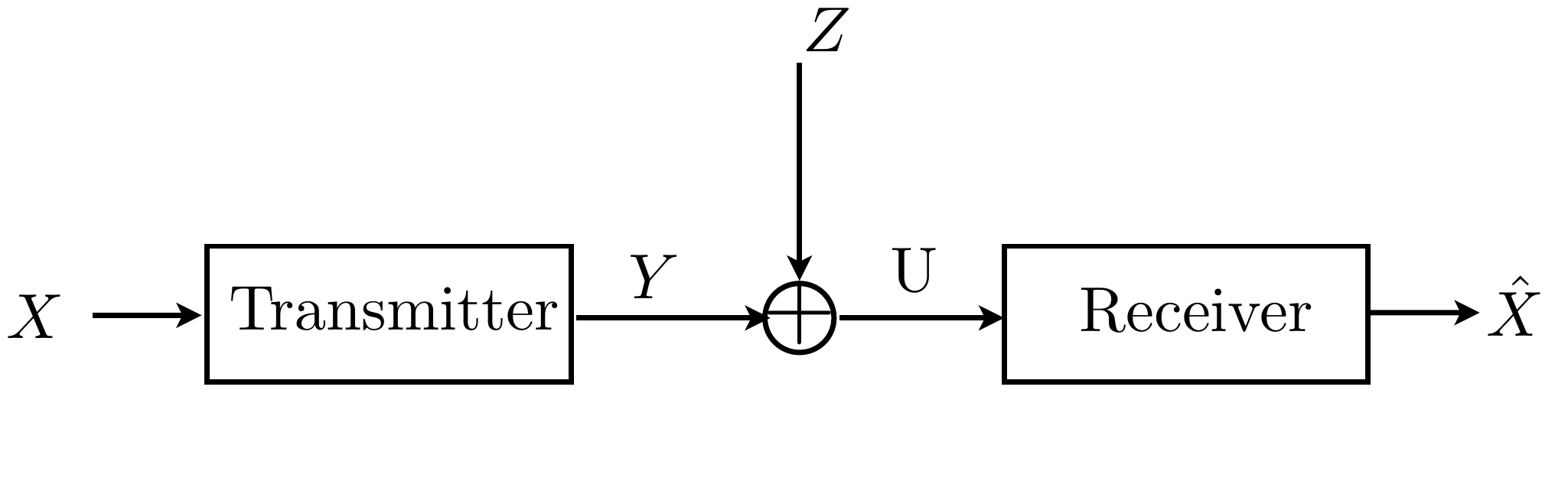}
\caption{The communication setting.}
\label{comm}
\end{figure}


\subsection{Communication Problem}
In \cite{mapping}, a communication scenario whose block diagram is shown in Figure 3 was studied.  In this setting, a scalar source ${X} \in\mathbb R$ is mapped into ${Y}\in  \mathbb R$ by function ${g}\in {\cal {S}}$, and transmitted over an additive noise channel. The channel output ${U}= Y+ Z $ is  mapped by the decoder to the estimate ${\hat X}$ via function ${h} \in {{\cal {S}}}$. The zero mean  noise $Z$ is assumed to be independent of the source $X$. The source density is denoted $f_X(\cdot)$ and the noise density is $f_Z(\cdot)$ with characteristic functions $F_X(\omega)$ and $F_Z(\omega)$, respectively. 

The objective is to minimize, over the choice of encoder ${g}$ and decoder ${h}$, the distortion
\begin{equation}
D={\mathbb E}\{({X}-{{ \hat X}})^2\},
\end{equation}
subject to the average transmission power constraint,
\begin{equation}
\mathbb E \{ {g^2}({X})\}.  \leq {P_T} \label{power_cons}.
\end{equation}

A result pertaining to the simultaneous linearity of optimal mappings is summarized in the next theorem. 

\begin{theorem}[\cite{mapping}]
The optimal mappings are either both linear or they are both nonlinear.
\label{simultenous}
\end{theorem}

The necessary and sufficient condition for linearity for both mappings is given by the following theorem.

\begin{theorem}[\cite{mapping}]
For a given power constraint $P_T$, noise $Z$ with variance $\sigma_Z^2$ and characteristic function $F_Z(\omega)$, source $X$ with variance $\sigma_X^2$ and characteristic function $F_X(\omega)$, the optimal encoder and decoder mappings are linear {\it if and only if}
\begin{equation}
\label{matching}
F_X  (\alpha \omega)=F_Z^\kappa (\omega)
\end{equation}
where $\kappa =\frac{P_T}{\sigma_Z^2}$ and $\alpha= \sqrt{ \frac{P_T}{\sigma_X^2}}$.
\end{theorem}

\subsection{Gaussian Jamming Problem}

The problem of transmitting independent and identically distributed Gaussian random variables over a Gaussian channel in the presence of an additive jammer was considered in \cite{basar1983gaussian,basar1985complete}. In \cite{basar1985complete} a game theoretic approach was developed and it was shown that the problem admits a mixed saddle point solution where the optimal transmitter and receiver employ a ``randomized" strategy. The randomization information can be sent over a side channel between transmitter and receiver or it could be viewed as the information generated by the third party  and observed by both transmitter and receiver \footnote{In practice, randomization can be achieved by (pseudo)random number generators at the transmitter and receiver using the same seed.}. Surprisingly, the optimal jamming strategy ignores the input to the jammer and merely generates  ``Gaussian" noise, independent of the source.  

\begin{theorem}[\cite{basar1983gaussian,basar1985complete}]
The optimal encoding function for the transmitter is randomized linear mapping: 
\begin{equation}
 Y(i)=\gamma(i) \alpha_T X(i), 
 \label{rand}
 \end{equation}
where $\gamma(i)$ is i.i.d. Bernoulli ($\frac{1}{2}$) over the alphabet $\{-1,1\}$
\begin{equation}\gamma(i)\sim Bern (\frac{1}{2}) \end{equation}
and $\alpha_T=\sqrt{\frac{P_T}{\sigma_X^2}} $.
The optimal jammer generates i.i.d. Gaussian output $Z(i)$ 
\begin{equation}Z(i)\sim \mathcal N (0, P_A). \end{equation}
 where $Z(i)$ is independent of  the source $X(i)$. The optimal receiver is 
 \begin{equation}
h({U}(i))=\frac{\sigma_X^2}{P_T+P_A+\sigma_N^2} {U}(i),
 \end{equation}
 and total cost is
  \begin{equation}
  J=\frac{\sigma_X^2(P_A+\sigma_N^2)}{P_T+P_A+\sigma_N^2}.
 \end{equation}
\label{th2}
\end{theorem}

\begin{remark}
In this paper, we study the generalized jamming problem which does not limit the set of sources and channels to Gaussian random variables. Surprisingly, as we show in Section V, the  linearity property of the optimal transmitter and receiver at the saddle point solution still holds, while the Gaussianity of the jammer output in the early special case was merely a means to satisfy this linearity condition, and does not hold in general. 
\label{rem}
\end{remark}

\begin{remark}
The proof of Theorem \ref{th2} relies on the fact that for a Gaussian source over a Gaussian channel, zero-delay linear mappings achieve the performance of the asymptotically high delay optimal source-channel communication system \cite{cover2006elements}. This fact is unique to the Gaussian source-channel pair, hence it might be tempting to conclude that the saddle point solution in Theorem \ref{th2} can only be obtained in the ``all Gaussian" setting. Perhaps surprisingly, in this paper, we show that there are infinitely many source-noise pairs that yield a saddle point solution similar to Theorem $\ref{th2}$ (see Theorem \ref{th} and Remark \ref{rem}).
\end{remark}

\section{A Simple Upper Bound Based on Linearity}

In this section, we present a new lemma that is used to upper bound the distortion of any zero-delay communication system by that of the fixed, best linear encoder and decoder. Although the main idea is quite simple, it is nevertheless  presented in a separate lemma, due to its operational significance here.

\begin{lemma}
Consider the problem setting in Figure \ref{jammingfig}. For any given jammer $f_Z(z)$, the distortion achievable by the transmitter-receiver, $D$,   is upper bounded by the distortion achieved by linear encoder and decoder  $D_L=\frac{\sigma_X^2(P_A+\sigma_N^2)}{P_T+P_A+\sigma_N^2}$ which is determined by second moments, regardless of the normalized densities. Hence, the linear mappings will maximize the distortion.  
\label{mainlemma}
\end{lemma}

\begin{proof}
Clearly, the the encoder and decoder can utilize the linear mappings that satisfy the power constraint $ P_T$ for any source and channel density. Hence, it is straightforward to achieve $D=D_L$ in any source-channel density by using linear mappings. 
\end{proof}

\begin{remark}
Lemma \ref{mainlemma} is the main result that connects the recent results on ``linearity" of optimal estimation and communication mappings to the jamming problem. Lemma \ref{mainlemma} implies that the optimal strategy for a jammer which can only control the density of the additive noise channel, is to force the transmitter and receiver to use linear mappings.
\end{remark}

\section{Main Result}
Our main results concerns the optimal strategy for the transmitter, the adversary and the receiver in Figure 1 for the transmission index $i$. Let us introduce a quantitiy $\beta$ as 
\begin{equation}
\beta \triangleq \frac{P_A+\sigma_N^2} {P_T}.
\end{equation}
{\bf Assumption}: Throughout this section, we assume that $F_X^{\beta}( \omega)$ is a valid characteristic function for a given $\beta \in \mathbb R^+$. The case where this assumption does not hold is analyzed in Section VII.

Next, we present our main result which pertains to optimal jamming. 

\begin{theorem}
For the jamming problem, the optimal encoding function for the transmitter is randomized  linear mapping: 
\begin{equation}
 Y(i)=\gamma(i) \alpha_T X(i), 
 \label{rand3}
 \end{equation}
where $\gamma(i)$ is i.i.d. Bernoulli ($\frac{1}{2}$) over the alphabet $\{-1,1\}$
\begin{equation}\gamma(i)\sim Bern (\frac{1}{2}) \end{equation}
and $\alpha_T=\sqrt{\frac{P_T}{\sigma_X^2}} $.
The optimal jamming function is to generate i.i.d.  output $Z(i)$ with characteristic function
 \begin{equation} F_Z(\omega)=\frac{F_X^{\beta}(\alpha_T \omega)}{ F_N(\omega)}\label{main} \end{equation}
 where $Z(i)$ is independent of  the adversarial  input $X(i)$.

 The optimal receiver is 
 \begin{equation}
h({U}(i))=\frac{\sigma_X^2}{P_T+P_A+\sigma_N^2} {U}(i),
\label{dec}
 \end{equation}
 and total cost is
  \begin{equation}
  J=\frac{\sigma_X^2(P_A+\sigma_N^2)}{P_T+P_A+\sigma_N^2}.
 \end{equation}
 
Moreover, this saddle point solution is (almost surely) unique. 
\label{th3}
\end{theorem}
\begin{proof}
We prove this result by verifying that the mappings in this theorem satisfy the saddle point inequalities given in (\ref{saddle}), following the approach in \cite{basar1986solutions}.

RHS of (\ref{saddle}): Suppose the policy of the jammer is given as in Theorem \ref{th3}. Then, the communication system at hand becomes identical to the communication problem considered in Section II.B, for which the linear encoder, i.e., $ Y(i)= \alpha_T X(i)$ is optimal (see Theorem 3). Any probabilistic encoder, given in the form of (\ref{rand3}) (irrespective of the density of $\gamma$) yields the same cost with deterministic encoders and hence is optimal. 

LHS of (\ref{saddle}): Let us derive the overall cost conditioned on the realization of the transmitter mappings (i.e., $\gamma=1$ and $\gamma=-1$) used in conjunction with optimal linear decoder, given in (\ref{dec}). If $\gamma=1$
\begin{equation}
D_1=J +\xi \mathbb E\{ZX\}+\psi \mathbb E\{ZN\}
\label{c1}
\end{equation}
  for some constants $\xi, \psi $, and similarly if  $\gamma=-1$
 \begin{equation}
D_2=J -\xi \mathbb E\{ZX\}-\psi  \mathbb E\{ZN\}
\label{c2}
\end{equation}
where the overall cost is
\begin{equation}
D(i)=\mathbb P(\gamma(i)=1)D_1+\mathbb P(\gamma(i)=-1)D_2 .
\end{equation} 
Clearly,  for $\gamma(i)\sim Bern (\frac{1}{2})$ overall cost $J$  is only a function of the second order statistics of the adversarial outputs, irrespective of the distribution of $Z$; hence the solution presented here is a saddle point.

Towards showing (almost sure) uniqueness, we start by restating the fact that the optimal solution for transmitter is in the randomized form given in (\ref{rand3}). Let us prove the properties which were not covered by the proof of the saddle point. 

Characteristic function of $Z$: The choice $F_Z(\omega)=\frac{F_X^{\beta}(\alpha_T \omega)}{ F_N(\omega)}$ renders the transmitter and receiver mapping linear, due to Theorem 3 and maximizes the overall cost due to Lemma 1. 

Independence of $Z$ of $X$ and $N$: If the jammer introduces some correlation, i.e., if $\mathbb E\{ZX\}\neq0$, the transmitter can adjust its Bernoulli parameter to decrease the distortion. Hence, the optimal adversarial strategy is setting $\mathbb E\{ZX\}=0$ which is guaranteed by the independence of zero mean random variables $Z$ and $X$. 

Choice of Bernoulli parameter: Note that the optimal choice of the Bernoulli parameter for the transmitters is $\frac{1}{2}$ since other choices will not cancel out  the cross terms  in (\ref{c1}) and (\ref{c2}). These cross terms can be exploited by the adversary to increase the cost, and hence an optimal strategy for transmitter is to set $\gamma=Bern(1/2)$.

\end{proof}

\begin{remark}
Note that Theorem \ref{th3} recovers the previous results that focus on the Gaussian source \cite{basar1983gaussian,basar1985complete}. When $X \sim {\mathcal N} (0, \sigma_X^2)$, then the unique matching noise, determined by (\ref{main}) is also Gaussian $Z \sim {\mathcal N} (0, P_A)$ for all power levels $P_A$ and $P_T$.  Hence, Theorem 5 strictly subsumes Theorem 4.  
\end{remark}

%
%

\section{Implications of the Main Result}
In this section, we explore some special cases obtained by varying $\beta$ and utilizing the matching condition (\ref{main}). We start with a simple but perhaps surprising result.
\begin{theorem}
In the case of identically distributed source and channel, i.e., $X \sim N$ and $P_T=P_A=\sigma_N^2$,  then optimal jamming strategy would be  generating a random variable identically distributed with $X$ (and $N$), and optimal transmitter functions are as given in Theorem \ref{th3}. \label{th}
\end{theorem}

\begin{proof} 
It is straightforward to see from (\ref{main}) that, at $\beta=2$,  the characteristic functions must be identical $F_Z(\omega)=F_X(\omega)$ almost everywhere. Since the characteristic function uniquely determines the density \cite{billingsley2008probability}, $Z\sim X$. 
\end{proof}

\begin{remark}
Theorem \ref{th} demonstrates that there is indeed a rich set of source and channel densities, that makes the optimal mappings linear. Hence, Gaussianity assumption of the source and channel is not necessary to achieve the saddle point solution. 
\end{remark}

Let us next consider a case where the jammer does not need to know the density of the source, i.e., can perform optimally regardless of the source density. 

\begin{theorem}
At asymptotically low CSNR level, i.e., $\beta \rightarrow \infty$, for a Gaussian channel, optimal jamming strategy is to generate Gaussian noise independent of the source, regardless of the source density. 

\end{theorem}
\begin{proof}
As we have shown in the proof of Theorem 4, the jammer's aim is to force the transmitter and the receiver to use linear mappings. Hence, the matching jamming noise (if exists) satisfies the following: 
\begin{equation}
F_Z(\omega)F_N(\omega)=F_X^{\beta}(\alpha_T\omega).
\label{upp}
\end{equation}
As $\beta \rightarrow \infty$, RHS of (\ref{upp}) converges to Gaussian characteristic function, due to central limit theorem \cite{billingsley2008probability}, and hence (\ref{main}) is asymptotically satisfied. 
\end{proof}

Another interesting case is the high CSNR level ($\beta\rightarrow 0$)  and Gaussian source case where any jammer output $Z$, independent of the source, is asymptotically optimal regardless of the noise density. 
\begin{theorem}
At an asymptotically high CSNR level, i.e., $\beta \rightarrow 0$, for a Gaussian source, optimal jamming strategy is to generate noise independent of the source regardless of the noise density. 
\end{theorem}
\begin{proof}
Again, the matching jamming noise (if exists) must satisfy 
\begin{equation}
(F_Z(\omega)F_N(\omega))^{\frac{1}{\beta}}=F_X(\alpha_T\omega).
\label{upp2}
\end{equation}
As $\beta \rightarrow 0$, LHS of (\ref{upp2}) converges to the Gaussian characteristic function and,  hence (\ref{main}) is asymptotically satisfied. 
\end{proof}

\section{The Non-matching Case}
What is the optimal jammer density $f_Z(\cdot)$,  when the jammer cannot make the optimal mappings linear, i.e., $F_X^{\beta}(\omega)$ is not a valid characteristic function? In the following, we first examine the case of the basic estimation setting, then extend our analysis to jamming setting.

\subsection{Estimation Setting}
The problem of interest is open, to our best knowledge, even in the more fundamental setting, i.e., for estimation problem depicted in Figure 2. Particularly, we are interested in the noise density $f_Z(\cdot)$ that maximizes minimum mean square error, $\mathbb E ((X-\mathbb E (X|U) )^2)$. Clearly, if $F_X^{\beta}(\omega)$ is a valid characteristic function, the worst-case noise will have the characteristic function $F_Z(\omega)=F_X^{\beta} (\omega)$ and make the optimal (MMSE) estimator linear. Intuitively, it is expected that in the case where  $F_X^{\beta}(\omega)$ is  {\it not} a valid characteristic function,  worst-case noise would be the one that forces the optimal estimator as close to linear as possible in some sense. In the following, we derive results, based on optimal estimation which show in what precise sense this intuition is correct. Let us restate, using Bayes' rule, the optimal estimator $h(u)=\mathbb E\{X|U=u\}$ as:
       \begin{equation}
        {h(u)}= \frac{\int {x}  \,  f_X( {x})  \,  f_{Z} ({u-x})\, {dx} } { \int  \,     f_X({ x})  \,  f_{Z}({u-x})  \,dx}
        \end{equation}
        which can also be written as:
           \begin{equation}
        {h(u)}=\frac{\int {F_X'(\omega)F_Z(\omega) e^{ju\omega}{d\omega} }} { \int  \,   F_X(\omega)F_Z(\omega)   e^{ju\omega}\,{d\omega}}  \label{a}
        \end{equation}
       We then replace $F_Z(\omega) $ and $F_X(\omega) $ with their polynomial expansions, particularly Gram-Charlier expansion over the Gaussian density (see e.g. \cite{cramer1999mathematical} for details): 
          \begin{align}
          F_Z(\omega)& = \sum \limits_{m=0}^{M} \left (1+\frac{\alpha_m }{m!}(jw)^m\right) e^{-\sigma_Z^2 w^2/2}\label{a1}  \\
           F_X(\omega) &= \sum \limits_{m=0}^{M} \left (1+\frac{\theta_m}{m!}(jw)^m \right) e^{-\sigma_X^2 w^2/2} \label{a2}
           \end{align}
           where $\alpha_m$ and $\theta_m$ are the polynomial coefficients associated with $ F_Z(\omega)$ and  $F_X(\omega)$ respectively. Observe that plugging (\ref{a1}) and (\ref{a2}) in (\ref{a}),   the  optimal estimator can be expressed as a ratio of two polynomials:
        \begin{align}
        {h(u)}= \frac{P_a(u)}{P(u)} .
        \end{align}
Let $P_n(u)$ be a sequence of polynomials orthonormal with respect to $P(u)$ (note that $P(u)$ is a probability density function, particularly it is $f_U(\cdot)$, i.e.,  the density of $U=X+Z$.) 
  \begin{equation} 
  \int {P_k(u)}{P_m(u)}{P(u)}du=\delta(m,k) \label{eqq}\end{equation}
  Then $h(u)$ can expanded in terms of ${P_m(u)}$
\begin{equation}
{h(u)}= \sum  \limits_{m=0}^{M}c_m{P_m(u)}
\end{equation}
                   where 
   \begin{equation}    
     c_m=\int P_m(u) P_a(u) du .
\end{equation}
        Then, MMSE  is
     \begin{align}      
     J=&\mathbb E((X-\mathbb E(X|U))^2)  \\
     =&\mathbb E(X^2)- (\mathbb E(X|U))^2 \label{eq1}\\
      =& \sigma_X^2- \sum \limits_{m=0}^{m} c_m^2.\label{eq3}
            \end{align}     
            where (\ref{eq1}) follows from the orthogonality principle, and (\ref{eq3}) follows from (\ref{eqq}). 
  Worst case noise  aims to maximize $J$ and hence minimize  $\sum  \limits_{m=0}^{M}c_m^2$. 
  Observe that $c_0=0$ and $c_1=\sqrt{ \frac{\sigma_X^2}{\sigma_X^2+\sigma_Z^2}}$. These two coefficients clearly are determined by the second order statistics of the source and the noise, while higher order coefficients, i.e., $c_m, m\geq2$ depend on the higher order statistics. Note also that the polynomials associated with these coefficients are $P_0(u)=1$ and $P_1(u)=\sqrt {\frac{\sigma_X^2}{\sigma_X^2+\sigma_Z^2}}u$. We present our main result regarding this setting. 
  
  \begin{lemma}
  The worst-case noise for the estimation setting minimizes $\sum \limits_{m=0}^{M} c_m^2$, where $c_m$ are the coefficients of the orthonormal polynomial expansion with measure $f_Y(\cdot)$. 
        \end{lemma}
        
Given the source density, we can find the optimal $M^{th}$ order polynomial approximation to the optimal estimator used in conjunction with the worst-case noise. In the following, we focus on finding the worst-case noise, that matches this estimator. 

Let us assume $h(u)= \sum  \limits_{m=0}^{M} b_m u^m$ for $b_m\in \mathbb R$. Then, the following holds: 

       \begin{equation}
\sum  \limits_{m=0}^{M}b_mu^m= \frac{\int {x}  \,  f_X( {x})  \,  f_{Z} ({u-x})\, {dx} } { \int  \,     f_X({ x})  \,  f_{Z}({u-x})  \,dx}.
\label{upeq1}
        \end{equation}
Expanding (\ref{upeq1}), and expressing integrals as convolutions, we have 
       \begin{equation}
\sum  \limits_{m=0}^{M} b_mu^m   (   f_X({ u})  \ast  f_{Z}({u})) =(uf_X(u))  \ast  f_{Z} ({u}) .
\label{upeq2}
        \end{equation}
        Taking the Fourier transforms of both sides, we obtain 
               \begin{equation}
\sum  \limits_{m=0}^{M} b_m  \frac{d^m}{d\omega^m} (   F_X({ \omega})  F_Z(\omega)) =F_X'(\omega)   F_{Z} ({\omega}) 
\label{upeq3}
        \end{equation}
Hence, given the optimal estimator, we can find the worst-case noise by solving the differential equation given in (\ref{upeq3}).

\subsection{Jamming Setting}
 Let us focus on the original problem of jamming. We carry a similar analysis to derive the best  $M^{th}$ order polynomial expansion of the decoder, given the encoder. For simplicity, we assume the the transmitter function is linear, i.e., as given (\ref{rand3}).  Note however that, as Theorem \ref{simultenous} implies, if the optimal decoder is nonlinear so must be the encoder. Hence, this approach will yield an approximate solution. 
 
 Towards deriving the optimal approximation of the decoder, we again express the decoder as 
 
   \begin{equation}
        {h(u)}= \frac{\int {x}  \,  f_X( {x})  \,  f_{N+Z} ({u-\alpha_Tx})\, {dx} } { \int  \,     f_X({ x})  \,  f_{N+Z}({u-\alpha_T x})  \,dx}
 \end{equation}
 where $\alpha_T=\sqrt{\frac{P_T}{\sigma_X^2}}$ and $f_{N+Z}$ is the density of $N+Z$. Nothing that $X$ and $Z$ are independent, we have 
         \begin{equation}
        {h(u)}=\frac{\int {F_X'(\alpha_T\omega)F_N(\omega) F_Z(\omega) e^{ju\omega}{d\omega} }} { \int  \,   F_X(\alpha_T\omega)F_N(\omega)F_Z(\omega)   e^{ju\omega}\,{d\omega}}  \label{a}
        \end{equation}
        which implies that, plugging the appropriate polynomial expression for $F_X(\omega)$, $F_N(\omega)$ and $F_Z(\omega)$, we can express $h(u)$ as the ratio of two polynomials, i.e., 
        \begin{align}
        {h(u)}= \frac{P_a(u)}{P(u)} .
                \end{align}
Again, expanding $h(u)$ the polynomials which are orthonormal under the measure $P(u)$ (which is the density of $\alpha_T X+Z+N$), and  following the same steps that led to (\ref{eq3}), we obtain 
\begin{equation}
J=\sigma_X^2-\sum \limits_{m=0}^{m} c_m^2
\end{equation}
where $c_m$'s are the coefficients of the polynomials that are orthonormal with respect to the density of the channel output $U=\alpha_T X+Z+N$. Hence, we can characterize the optimal jamming density, as the one that minimizes $\sum \limits_{m=0}^{m} c_m^2$. Similar to estimation setting, once the best polynomial approximation is found, the optimal jamming density can be obtained by solving a differential equation which can be obtained following the same steps that yield (\ref{upeq3}). 



\section{Discussion}
 In this paper, we studied the problem of optimal zero-delay jamming over an additive noise channel. Utilizing the recent results on conditions for linearity of optimal estimation, and of optimal mappings in source-channel coding, we obtained the saddle-point solution to the jamming problem for general sources and channels. We showed that linearity is essential in the  jamming problem, in the sense that the optimal jamming strategy is to effectively force both transmitter and receiver to linear mappings. We analyzed conditions and general settings where such strategy can indeed be achieved by the jammer, and provided a ``matching condition"  which strictly subsumes the prior results specialized to all Gaussian settings.  Finally, we provided a procedure to approximate optimal jamming in the cases where the jammer cannot impose linearity on the transmitter and the receiver.

Analysis in this paper is limited to scalar sources and channels. An important extension of this work, currently under investigation, will be on vector sources and/or channels. Another line of future work involves the precise characterization of jamming noise in the non-matching case. 



\bibliographystyle{IEEEbib}

\bibliography{ref}

\end{document}